\documentclass{article}

\usepackage{arxiv}
\usepackage{float}
\usepackage[utf8]{inputenc} 
\usepackage[T1]{fontenc}    
\usepackage{hyperref}       
\usepackage{url}            
\usepackage{booktabs}       
\usepackage{amsfonts}       
\usepackage{nicefrac}       
\usepackage{microtype}      
\usepackage{lipsum}
\usepackage{graphicx}
\graphicspath{ {./images/} }
\usepackage{amssymb}
\usepackage{amsmath}
\usepackage{lineno,hyperref}
\usepackage{latexsym}
\usepackage{amsthm}
\usepackage{amssymb}
\usepackage{amsmath,float,cancel}
\usepackage{ulem}
\usepackage{tikz}
\usepackage{tkz-graph}
\usetikzlibrary{decorations,arrows,shapes}
\usepackage{enumerate}
\usepackage{comment}

\newtheorem{theorem}{Theorem}[section]
\newtheorem{lemma}[theorem]{Lemma}
\newtheorem{conjecture}[theorem]{Conjecture}
\newtheorem{remark}{Remark}
\newtheorem{proposition}[theorem]{Proposition}

\newtheorem{corollary}[theorem]{Corollary}

\newcommand{\chila}{\chi_{\mathrm{la}}}

\newcommand{\Cat}{\operatorname{Cat}}

\title{Counterexamples to the local antimagic leaf conjecture}

\author{
 André Ebling Brondani \\
  Universidade Federal Fluminense - UFF\\
  Volta Redonda, RJ\\
  \texttt{andrebrondani@id.uff.br} \\
   \And
 Francisca Andrea Macedo França \\
  Universidade Federal Fluminense - UFF\\
  Volta Redonda, RJ\\
  \texttt{francisca\_franca@id.uff.br} \\
}

\begin{document}
\maketitle
\begin{abstract}
It has been conjectured that $\chila(T)\leq\ell(T)+2$ for every tree $T$ of order at least three, where $\chila(T)$ is the local antimagic chromatic number and $\ell(T)$ denotes the number of leaves. We disprove this bound and show that no absolute constant can replace $2$ in an inequality of this form. Our approach uses degree-based intervals containing all possible induced vertex weights under bijective edge labelings, together with a separation criterion that forces selected internal vertex weights to be distinct. We first construct caterpillars $T_3$ and $T_4$ satisfying $\chila(T_3)=\ell(T_3)+3$ and $\chila(T_4)=\ell(T_4)+4$, respectively. We then construct, for every positive integer $r$, infinitely many pairwise nonisomorphic caterpillars with exactly $r$ internal vertices for which every bijective edge labeling is antimagic. Each such caterpillar $T$ satisfies $\chila(T)=|V(T)|=\ell(T)+r$. Consequently, $\chila(T)-\ell(T)$ is unbounded even within the class of caterpillars, and no absolute constant $C$ can satisfy $\chila(T)\leq\ell(T)+C$ for every tree $T$.
\end{abstract}


\section{Introduction}\label{sec:introduction}

Let $G=(V(G),E(G))$ be a finite simple graph with $m$ edges. According to Hartsfield and Ringel~\cite{HartsfieldRingel1990}, a bijection $f:E(G)\to\{1,2,\ldots,m\}$ is called an \emph{antimagic labeling} of $G$ when the induced vertex weights $f^{+}(v)$, obtained by summing the labels of the edges incident with $v$, are pairwise distinct over all vertices of $G$. Hartsfield and Ringel conjectured that every connected graph of order at least three is antimagic. A weaker version of this concept is the \emph{local antimagic labeling}, which requires distinct induced weights only at adjacent vertices. These labelings were introduced independently by Arumugam et al.~\cite{Arumugam2017} and Bensmail, Senhaji and Lyngsie~\cite{BSL2017}. More precisely, a bijection $f:E(G)\to\{1,2,\ldots,m\}$ is a \emph{local antimagic labeling} of $G$ if $f^{+}(u)\neq f^{+}(v)$ whenever $uv\in E(G)$. Such a labeling induces a proper vertex coloring of $G$, and the minimum number of distinct induced vertex weights over all local antimagic labelings of $G$ is called the \emph{local antimagic chromatic number}, denoted by $\chila(G)$. Haslegrave~\cite{Has2018} subsequently proved that every connected graph other than $K_2$ admits a local antimagic labeling. Since then, exact values and constructive bounds for $\chila(G)$ have been obtained for several graph classes; see, for instance,~\cite{LNS2020,NSW23,Ventura2024}.

Although every antimagic labeling is local antimagic, the two notions lead to different problems. The classical antimagic problem asks for the existence of a labeling that distinguishes all vertices, whereas the local antimagic chromatic problem asks how few distinct induced vertex weights can occur while adjacent vertices remain distinguished.
In particular, the existence of an antimagic labeling does not imply that $\chila(G)=|V(G)|$, since another local antimagic labeling may assign the same induced weight to nonadjacent vertices. This distinction will be important in the results developed here.

For a tree $T$ of order at least three, let $\ell(T)$ denote its number of leaves. Arumugam et al.~\cite{Arumugam2017} established the fundamental lower bound $\chila(T)\geq\ell(T)+1$. It has been conjectured that $\chila(T)\leq\ell(T)+2$, which, together with the leaf lower bound, would restrict the parameter to the two values $\ell(T)+1$ and $\ell(T)+2$; see~\cite{Baca22,PACT2020}.
Throughout this paper, we refer to this conjecture as the \emph{leaf conjecture}.

A number of results support this prediction for restricted families of
trees. Arumugam et al.~\cite{Arumugam2017} determined the parameter
for paths and stars, both of which attain the leaf lower bound.
Premalatha et al.~\cite{PACT2020} studied several further classes of
trees and obtained values consistent with the two possibilities
predicted by the conjecture. Ba\v{c}a et al.~\cite{Baca22} verified
the conjectured upper bound for complete full $t$-ary trees and, in
particular, proved that the lower bound is attained when $t$ is odd.
Lau, Shiu and Soo~\cite{LauShiuSoo} investigated spider graphs and
obtained bounds and exact values for several subclasses. Thus the
values $\ell(T)+1$ and $\ell(T)+2$ arise naturally in structurally
different families of trees, which helps explain the plausibility of
the conjecture in light of the known positive cases.

The purpose of the present paper is to show that this behavior is not
universal, even within the class of caterpillars. Our approach is based
on elementary degree bounds for induced vertex weights. If a vertex
has degree $d$ in a graph with $m$ edges, then under any bijective edge
labeling its induced weight lies between $d(d+1)/2$ and
$d(2m-d+1)/2$. These weight intervals provide a mechanism for
separating induced weights at nonadjacent vertices, while adjacency
itself guarantees distinct weights under a local antimagic labeling.
Combining these observations yields a separation criterion that forces
prescribed internal vertices of a tree to have pairwise distinct
induced weights, all larger than every possible leaf weight.

We first apply this criterion to two explicit caterpillars. The first, $T_3$, has $40$ leaves and satisfies $\chila(T_3)=43=\ell(T_3)+3$, thereby disproving the leaf conjecture. The second, $T_4$, has $240$ leaves and satisfies $\chila(T_4)=244=\ell(T_4)+4$, showing that replacing the constant $2$ by $3$ still does not yield a universal upper bound. We then give an explicit construction showing that, for every positive integer $r$, there exist infinitely many pairwise nonisomorphic caterpillars with exactly $r$ internal vertices for which every bijective edge labeling is antimagic. In particular, every local antimagic labeling of such a caterpillar induces pairwise distinct vertex weights. Consequently, $\chila(T)=|V(T)|=\ell(T)+r$. Thus the construction also provides, for every prescribed number of internal vertices, infinitely many trees whose local antimagic chromatic number attains the largest possible value. Moreover, $\chila(T)-\ell(T)$ is unbounded even within the class of caterpillars, and hence there is no absolute constant $C$ such that $\chila(T)\leq\ell(T)+C$ for every tree $T$.

This labeling-independent property also connects the main construction
to the classical notion of antimagic labeling introduced by Hartsfield
and Ringel. Although $\chila(G)$ is defined through the weaker local
condition, the degree structure of the caterpillars constructed here
forces all induced vertex weights to be distinct under every bijective
edge labeling. Thus, for this family, the largest possible value
$\chila(T)=|V(T)|$ is not a consequence of a specially chosen labeling,
but of a structural obstruction that prevents any two vertices from
receiving the same induced weight.

These results reinforce the importance of determining the exact local antimagic chromatic number for specific classes of trees. Although no universal upper bound of the form $\ell(T)+C$ holds with an absolute constant $C$, the bounds established for
particular families remain valid. This motivates the search for structural conditions ensuring $\chila(T)\leq\ell(T)+2$ and, more generally, for classes in which
$\chila(T)-\ell(T)$ is uniformly bounded.

The remainder of the paper is organized as follows.
Section~\ref{sec:preliminaries} fixes the notation and recalls the leaf lower bound and the leaf conjecture. Section~\ref{sec:degree-bounds} establishes degree bounds for induced vertex weights and derives the separation criterion. Section~\ref{sec:counterexamples} contains the two explicit counterexamples. In Section~\ref{sec:unbounded}, we construct the general family and prove that $\chila(T)-\ell(T)$ is unbounded even within the class of caterpillars. Finally, Section~\ref{sec:conclusion} discusses consequences and further questions.

\section{Preliminaries}\label{sec:preliminaries}

We use standard graph-theoretic notation. All graphs considered in this paper are finite, simple and connected, with at least three vertices. Let $G=(V(G),E(G))$ be a graph with size $m=|E(G)|$, and write $[m]=\{1,2,\ldots,m\}$. For a vertex $v\in V(G)$, let $E_G(v)$ denote the set of edges incident with $v$, and let $d_G(v)=|E_G(v)|$ be its degree. The subscript is omitted whenever the graph is clear from the context. For a tree $T$, a vertex of degree one is called a \emph{leaf}, and every other vertex is called an \emph{internal vertex}. We denote by
$\ell(T)$ the number of leaves of $T$. A \emph{caterpillar} is a tree whose internal vertices induce a path, possibly a trivial one.

Given a bijection $f:E(G) \to [m]$, the \emph{induced weight} of vertex $v\in V(G)$ is defined by 
\begin{equation*}
f^{+}(v)=\sum_{e\in E_G(v)}f(e).    
\end{equation*}

Following Hartsfield and Ringel~\cite{HartsfieldRingel1990}, the labeling $f$ is called \emph{antimagic} if the induced vertex weights are pairwise distinct, that is, $f^{+}(u) \neq f^{+}(v)$ for every pair of distinct vertices $u,v\in V(G)$. A graph is called
\emph{antimagic} if it admits an antimagic labeling. The local version, introduced by Arumugam et al.~\cite{Arumugam2017} and Bensmail, Senhaji and Lyngsie~\cite{BSL2017}, requires only that $f^{+}(u)\neq f^{+}(v)$ whenever $uv\in E(G)$; in this case, $f$ is called a \emph{local antimagic labeling} of $G$.

For any bijection $f:E(G)\to[m]$, we write
\begin{equation*}
C(f)=\{f^{+}(v):v\in V(G)\}
\end{equation*}
for the set of distinct induced vertex weights of $f$. The local antimagic chromatic number of $G$ is
\begin{equation*}
\chila(G)=\min_f |C(f)|,
\end{equation*}
where the minimum is taken over all local antimagic labelings $f$ of $G$. Haslegrave~\cite{Has2018} proved that every connected graph of order at least three admits a local antimagic labeling. Hence $\chila(G)$ is
well defined for all graphs considered in this paper.

We recall the standard lower bound established in~\cite{Arumugam2017}.

\begin{theorem}[\cite{Arumugam2017}]\label{thm:leaf-lower-bound}
Let $T$ be a tree of order at least three with $\ell(T)$ leaves. Then
\begin{equation*}
\chila(T)\geq\ell(T)+1.
\end{equation*}
\end{theorem}

We shall refer to the following conjecture as the \emph{leaf conjecture}; see~\cite{Baca22,PACT2020}.

\begin{conjecture}\label{conj:leaf}
Let $T$ be a tree of order at least three with $\ell(T)$ leaves. Then
\begin{equation*}
\chila(T)\leq\ell(T)+2.
\end{equation*}
\end{conjecture}

Together with Theorem~\ref{thm:leaf-lower-bound}, Conjecture~\ref{conj:leaf} is equivalent to $\chila(T)\in\{\ell(T)+1,\ell(T)+2\}$.

For later use, we introduce the following notation for caterpillars. For $r\geq2$, let $\Cat(p_1,\ldots,p_r)$ denote the tree obtained from the path $P=(v_1,\ldots,v_r)$ by adjoining $p_i$ leaves to $v_i$. We assume $p_1,p_r\geq1$ and $p_i\geq0$ for $2\leq i\leq r-1$, so that $P$ is precisely the path induced by the internal vertices. If $d_i=d_T(v_i)$, then
\begin{equation*}
\ell(T)=\sum_{i=1}^r p_i,\qquad |V(T)|=\ell(T)+r,\qquad |E(T)|=\ell(T)+r-1,
\end{equation*}
and equivalently
\begin{equation}\label{eq:caterpillar-degrees}
|E(T)|=\sum_{i=1}^r d_i-r+1.
\end{equation}

\section{Degree bounds for induced vertex weights}\label{sec:degree-bounds}

For a graph with $m$ edges and an integer $1\leq d\leq m$, define
\begin{equation*}
L_m(d)=\frac{d(d+1)}{2}\qquad\text{and}\qquad U_m(d)=\frac{d(2m-d+1)}{2},
\end{equation*}
and write
\begin{equation*}
I_m(d)=[L_m(d),U_m(d)].
\end{equation*}

\begin{lemma}\label{lem:weight-bounds}
Let $G$ be a graph with $m$ edges and let $f:E(G)\to[m]$ be a bijection. If $d_G(v)=d$, then
\begin{equation*}
L_m(d)\leq f^{+}(v)\leq U_m(d).
\end{equation*}
\end{lemma}

\begin{proof}
Let $t_1<\cdots<t_d$ be the labels of the edges incident with $v$. Since they are distinct integers in $[m]$, we have $i\leq t_i\leq m-d+i$ for every $1\leq i\leq d$. Hence
\begin{equation*}
\sum_{i=1}^d i\leq f^{+}(v)=\sum_{i=1}^d t_i\leq\sum_{i=1}^d(m-d+i),
\end{equation*}
and the two extreme sums are precisely $L_m(d)$ and $U_m(d)$.
\end{proof}

Theorem~\ref{thm:leaf-lower-bound} guarantees that every local antimagic labeling of a tree produces at least one induced vertex weight in addition to the $\ell(T)$ distinct leaf weights. The degree bounds in Lemma~\ref{lem:weight-bounds} allow this lower bound to be strengthened when suitable internal vertices can be separated by their possible
weight intervals. Indeed, these intervals can separate selected
internal vertices from all leaves, while adjacency and disjoint induced-weight intervals provide two complementary mechanisms for distinguishing the selected vertices from one another. This leads to the following refinement of the lower bound of Arumugam et al.~\cite{Arumugam2017}.

\begin{proposition}\label{prop:separation-criterion}
Let $T$ be a tree with $m\geq2$ edges, and let $S$ be a set of
internal vertices. Suppose that
\begin{enumerate}[(i)]
\item[(i)] $L_m(d_T(v))>m$ for every $v\in S$;
\item[(ii)] $I_m(d_T(u))\cap I_m(d_T(v))=\varnothing$ for every pair
of distinct nonadjacent vertices $u,v\in S$.
\end{enumerate}
Then every local antimagic labeling $f$ of $T$ satisfies
\begin{equation}\label{eq:separation-lower-bound}
|C(f)|\geq\ell(T)+\max\{1,|S|\}.
\end{equation}
Consequently,
\[
\chila(T)\geq\ell(T)+\max\{1,|S|\}.
\]
\end{proposition}

\begin{proof}
Let $f$ be a local antimagic labeling of $T$.
Since the induced weight of a leaf is exactly the label of its unique
incident edge, the $\ell(T)$ leaves have pairwise distinct induced
weights in $[m]$. Since $m\geq2$, the edge labeled $m$ has at least
one internal endpoint $w$. Thus $f^+(w)>m$, which gives
$|C(f)|\geq\ell(T)+1$.

For each $v\in S$, Lemma~\ref{lem:weight-bounds} and
condition~\textnormal{(i)} give
\[
f^{+}(v)\geq L_m(d_T(v))>m.
\]
Hence no vertex of $S$ has the same induced weight as a leaf.
Moreover, let $u,v\in S$ be distinct. If $uv\in E(T)$, then
$f^{+}(u)\neq f^{+}(v)$ because $f$ is local antimagic.
Otherwise, their induced weights belong to disjoint intervals
by condition~\textnormal{(ii)}. Therefore, the vertices of $S$
contribute $|S|$ pairwise distinct weights, all different from
the leaf weights, and
$|C(f)|\geq\ell(T)+|S|$.

Combining the two bounds proves~\eqref{eq:separation-lower-bound}. Taking the minimum over all
local antimagic labelings of $T$ yields the final assertion.
\end{proof}

\begin{corollary}\label{cor:maximum-weights}
Let $T$ be a tree with $m\geq2$ edges. Suppose that the set of all internal vertices satisfies the hypotheses of Proposition~\ref{prop:separation-criterion}. Then
$\chila(T)=|V(T)|$. If, in addition, $I_m(d_T(u))\cap I_m(d_T(v))=\varnothing$ for every pair of distinct internal vertices $u,v$, then every bijection $f:E(T)\to[m]$ is antimagic.
\end{corollary}
\begin{proof}
If $S$ is the set of all internal vertices, Proposition~\ref{prop:separation-criterion} gives $\chila(T)\geq\ell(T)+|S|=|V(T)|$. Since a local antimagic labeling exists and can induce at most $|V(T)|$ distinct vertex weights, equality follows. Under the additional hypothesis, take an arbitrary bijection $f:E(T)\to[m]$. The internal weights lie above $m$ in pairwise disjoint intervals, whereas the leaf weights are pairwise distinct and at most $m$. Hence all induced vertex weights are pairwise distinct, so $f$ is antimagic.
\end{proof}

\section{Counterexamples to the leaf conjecture}\label{sec:counterexamples}

We first apply the separation criterion to two caterpillars. The first disproves the leaf conjecture, while the second shows that even the upper bound $\ell(T)+3$ does not hold for arbitrary trees.

Let $T_3=\Cat(8,7,25)$, whose internal vertices induce the path $P=(a,b,c)$. Then
\begin{equation*}
\ell(T_3)=40,\qquad |E(T_3)|=42,\qquad |V(T_3)|=43,
\end{equation*}
and $d(a)=d(b)=9$, whereas $d(c)=26$.
Figure~\ref{fig:T3} displays the tree $T_3$.

\begin{figure}[h!]
\centering
\begin{tikzpicture}[
    scale=1,
    hub/.style={circle,fill=black,inner sep=2.1pt},
    leaf/.style={circle,fill=black,inner sep=.85pt}
]
\coordinate (a) at (-4,0);
\coordinate (b) at (-1.2,0);
\coordinate (c) at (3.2,0);
\draw[thick] (a)--(b)--(c);
\foreach \i in {1,...,8}{
    \pgfmathsetmacro{\ang}{-165+(\i-1)*150/7}
    \path (a) ++(\ang:1.25) coordinate (xa\i);
    \draw (a)--(xa\i);
    \node[leaf] at (xa\i) {};
}
\foreach \i in {1,...,7}{
    \pgfmathsetmacro{\ang}{-165+(\i-1)*150/6}
    \path (b) ++(\ang:1.25) coordinate (xb\i);
    \draw (b)--(xb\i);
    \node[leaf] at (xb\i) {};
}
\foreach \i in {1,...,25}{
    \pgfmathsetmacro{\ang}{-165+(\i-1)*150/24}
    \path (c) ++(\ang:2.0) coordinate (xc\i);
    \draw (c)--(xc\i);
    \node[leaf] at (xc\i) {};
}
\node[hub,label=above:{$a$}] at (a) {};
\node[hub,label=above:{$b$}] at (b) {};
\node[hub,label=above:{$c$}] at (c) {};
\end{tikzpicture}
\caption{The caterpillar $T_3$ with $43$ vertices.}
\label{fig:T3}
\end{figure}
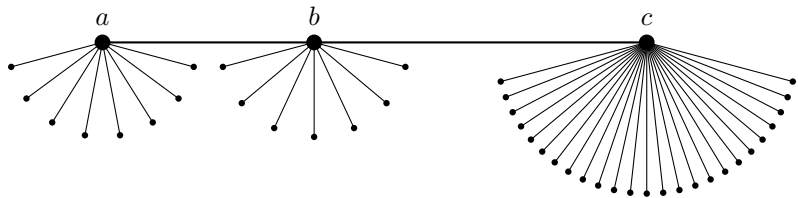

The caterpillar $T_3$ has three internal vertices, which is the minimum number possible for a counterexample to the leaf conjecture. Indeed, every tree with at most two internal vertices satisfies $\chila(T)\leq |V(T)|\leq \ell(T)+2$. The following theorem shows that the leaf conjecture already fails for a tree with the minimum possible number of internal vertices.

\begin{theorem}\label{thm:T3}
The caterpillar $T_3$ satisfies
\begin{equation*}
\chila(T_3)=43=\ell(T_3)+3.
\end{equation*}
\end{theorem}
\begin{proof}
For $m=42$, the relevant induced-weight intervals are
\begin{equation*}
I_{42}(9)=[45,342],\qquad
I_{42}(26)=[351,767].
\end{equation*}
Both intervals lie above $42$, and they are disjoint. Since $a$ and $c$ are the only nonadjacent internal
vertices, the set $S=\{a,b,c\}$ satisfies the hypotheses of Proposition~\ref{prop:separation-criterion}. Therefore, Corollary~\ref{cor:maximum-weights} gives
\[
\chila(T_3)=|V(T_3)|=43=\ell(T_3)+3.
\]
\end{proof}


We next give a second counterexample showing that even the bound
$\ell(T)+3$ is not universal.

Let $T_4=\Cat(21,20,99,100)$, whose internal vertices induce the path $P=(a,b,c,d)$. Then
\begin{equation*}
\ell(T_4)=240,\qquad |E(T_4)|=243,\qquad |V(T_4)|=244,
\end{equation*}
with $d(a)=d(b)=22$ and $d(c)=d(d)=101$.

\begin{theorem}\label{thm:T4}
The caterpillar $T_4$ satisfies
\begin{equation*}
\chila(T_4)=244=\ell(T_4)+4.
\end{equation*}
\end{theorem}
\begin{proof}
For $m=243$, the relevant induced-weight intervals are
\[
I_{243}(22)=[253,5115],\qquad
I_{243}(101)=[5151,19493].
\]
They lie above $243$ and are disjoint. Since the nonadjacent pairs of internal vertices are $(a,c)$, $(a,d)$, and $(b,d)$, the set $S=\{a,b,c,d\}$ satisfies the hypotheses of Proposition~\ref{prop:separation-criterion}. Therefore, Corollary~\ref{cor:maximum-weights} yields
\[
\chila(T_4)=|V(T_4)|=244=\ell(T_4)+4.
\]
\end{proof}

\begin{remark}\label{rem:minimal-internal-count}
The examples use the smallest possible numbers of internal vertices for their respective violations. Indeed, if a tree has $r$ internal vertices, then $\chila(T)\leq|V(T)|=\ell(T)+r$. Thus violating the bound $\ell(T)+2$ requires $r\geq3$, while violating $\ell(T)+3$ requires $r\geq4$. This observation does not assert that either $T_3$ or $T_4$ is of minimum order among such counterexamples.
\end{remark}

\section{Caterpillars with unbounded $\chila(T)-\ell(T)$}\label{sec:unbounded}

The preceding examples show that the bounds $\ell(T)+2$ and $\ell(T)+3$ are not universal upper bounds. We now prove a stronger statement: the difference $\chila(T)-\ell(T)$ is unbounded even among caterpillars. The construction is explicit and, for each prescribed number of internal vertices, produces infinitely many pairwise nonisomorphic examples.

\subsection{Degree-separated caterpillars}\label{subsec:degree-separated}

Let $T$ be a caterpillar whose internal vertices induce the path $P=(v_1,\ldots,v_r)$. Let $m=|E(T)|$ and write
$d_i=d_T(v_i)$ for $1\leq i\leq r$. We say that $T$ is
\emph{degree-separated} if $L_m(d_i)>m$ for every $1\leq i\leq r$ and $I_m(d_i)\cap I_m(d_j)=\varnothing$ whenever $|i-j|\geq2$. Thus, under every bijective edge labeling, every internal vertex has induced weight greater than every leaf weight. Moreover, the possible induced-weight intervals of any two nonadjacent internal vertices are disjoint, whereas adjacent internal vertices already have distinct
induced weights under a local antimagic labeling.

\begin{proposition}\label{prop:degree-separated-caterpillar}
Let $T$ be a degree-separated caterpillar with $r$ internal vertices. Then
\begin{equation*}
\chila(T)=|V(T)|=\ell(T)+r.
\end{equation*}
\end{proposition}

\begin{proof}
The set of all internal vertices satisfies the hypotheses of Proposition~\ref{prop:separation-criterion}. Hence Corollary~\ref{cor:maximum-weights} gives $\chila(T)=|V(T)|$. Since $T$ has $\ell(T)$ leaves and exactly $r$ internal vertices, $|V(T)|=\ell(T)+r$.
\end{proof}

\subsection{An explicit infinite family}\label{subsec:explicit-family}

We now construct infinitely many degree-separated caterpillars.
The construction is guided by the estimates $L_m(d)\geq d^2/2$ and $U_m(d)\leq md$. We seek internal degrees $d_1<\cdots<d_r$ such that $L_m(d_1)>m$ and $U_m(d_i)<L_m(d_{i+1})$ for $1\leq i<r$. These inequalities ensure that all internal vertex weights exceed the leaf weights and that the corresponding induced-weight intervals are pairwise disjoint. To obtain them, we choose a sequence whose sum can be controlled by its largest term and whose successive terms satisfy suitable quadratic identities.

Fix an integer $r\geq2$ and define
\begin{equation}\label{eq:base-degrees}
a_i=4^{\,2^r-2^{r-i}}\quad(1\leq i\leq r),\qquad A=a_r=4^{\,2^r-1}.
\end{equation}
These positive integers satisfy
\begin{equation}\label{eq:degree-identities}
a_1^2=4A,\qquad a_{i+1}^2=4Aa_i\quad(1\leq i<r).
\end{equation}
Moreover,
\begin{equation*}
\frac{a_{i+1}}{a_i}=4^{\,2^{r-i-1}}\geq4,
\end{equation*}
so $a_1<\cdots<a_r=A$ and
\begin{equation}\label{eq:degree-sum}
\sum_{i=1}^r a_i\leq A\sum_{j=0}^{r-1}4^{-j}<\frac{4A}{3}.
\end{equation}

For any integer $t\geq1$, put $d_i=ta_i$. Let $H_{r,t}$ be the caterpillar obtained from the path $P=(v_1,\ldots,v_r)$ by adjoining $d_1-1$ leaves to $v_1$, $d_r-1$ leaves to $v_r$, and $d_i-2$ leaves to $v_i$ for $2\leq i\leq r-1$. Since $d_i\geq a_1\geq16$, all attachment numbers are positive, and the internal vertices of $H_{r,t}$ are precisely $v_1,\ldots,v_r$. Hence they induce the path $P$. By~\eqref{eq:caterpillar-degrees} and~\eqref{eq:degree-sum}, if $m=|E(H_{r,t})|$, then
\begin{equation}\label{eq:family-size}
m=t\sum_{i=1}^r a_i-r+1<\frac{4tA}{3}.
\end{equation}
Furthermore,
\begin{equation}\label{eq:family-order-leaves}
|V(H_{r,t})|=t\sum_{i=1}^r a_i-r+2,\qquad \ell(H_{r,t})=t\sum_{i=1}^r a_i-2r+2.
\end{equation}

\begin{theorem}\label{thm:main}
For every integer $r\geq1$, there exist infinitely many pairwise nonisomorphic caterpillars $H$ with exactly $r$ internal vertices such that every bijection $f:E(H)\to\{1,\ldots,|E(H)|\}$ is antimagic. In particular,
\begin{equation*}
\chila(H)=|V(H)|=\ell(H)+r.
\end{equation*}
\end{theorem}
\begin{proof}
For $r=1$, take $H_{1,t}=K_{1,t+1}$ for $t\geq1$. Under any bijective edge labeling, the leaf weights are the distinct integers $1,\ldots,t+1$, whereas the center has induced weight $(t+1)(t+2)/2>t+1$. Hence every bijection is antimagic, and varying $t$ gives infinitely many pairwise nonisomorphic stars.

Now let $r\geq2$ and consider $H_{r,t}$ defined above. We first show that the induced weight of every internal vertex exceeds $m$. By~\eqref{eq:degree-identities} and~\eqref{eq:family-size},
\begin{equation}\label{eq:family-above-leaves}
L_m(d_1)>\frac{d_1^2}{2}=\frac{t^2a_1^2}{2}=2t^2A\geq2tA>\frac{4tA}{3}>m.
\end{equation}
Since $L_m(d)$ is increasing in $d$, it follows that $L_m(d_i)>m$ for every $1\leq i\leq r$.

Next, for $1\leq i<r$, we have $U_m(d_i)\leq d_i m$. Therefore, using~\eqref{eq:degree-identities} and~\eqref{eq:family-size},
\begin{equation*}
U_m(d_i)<\frac{4t^2Aa_i}{3}=\frac{t^2a_{i+1}^2}{3}<\frac{t^2a_{i+1}^2}{2}<L_m(d_{i+1}).
\end{equation*}
Hence the intervals $I_m(d_1),\ldots,I_m(d_r)$ are pairwise disjoint. Together with \eqref{eq:family-above-leaves}, Corollary~\ref{cor:maximum-weights} shows that every bijection $f:E(H_{r,t})\to[m]$ is antimagic. Consequently,
\begin{equation*}
\chila(H_{r,t})=|V(H_{r,t})|=\ell(H_{r,t})+r.
\end{equation*}
Finally, for fixed $r$, the order in \eqref{eq:family-order-leaves} is strictly increasing with $t$, so the trees $H_{r,t}$ are pairwise nonisomorphic.
\end{proof}

\begin{remark}
Since $\chila(G)\leq |V(G)|$ for every graph $G$ considered here,
the caterpillars $H_{r,t}$ have maximum possible local antimagic
chromatic number. In particular, for every positive integer $r$,
there exist infinitely many pairwise nonisomorphic trees with exactly
$r$ internal vertices satisfying $\chila(T)=|V(T)|$. Thus trees with
maximum local antimagic chromatic number may have arbitrarily many
internal vertices.
\end{remark}

\begin{corollary}\label{cor:no-constant}
There is no absolute constant $C$ such that
\begin{equation*}
\chila(T)\leq\ell(T)+C
\end{equation*}
for every tree $T$ of order at least three.
\end{corollary}

\begin{proof}
Given $C$, choose a positive integer $r>C$ and apply Theorem~\ref{thm:main}. The resulting caterpillars satisfy 
\begin{equation*}
\chila(T)-\ell(T)=r>C.    
\end{equation*}
\end{proof}

\begin{remark}\label{rem:construction-size}
The base sequence $(a_1,a_2,\ldots,a_r)$ in~\eqref{eq:base-degrees} is chosen
to yield a uniform explicit construction rather than small examples. For $t=1$, the first two instances with at least two internal vertices are $H_{2,1}=\Cat(15,63)$ and $H_{3,1}=\Cat(255,4094,16383)$. In particular, $H_{3,1}$ has $20735$ vertices, whereas $T_3$ has only $43$. The smaller counterexamples exploit adjacency to allow equal degrees at neighboring internal vertices, while the general construction separates the induced-weight intervals determined by the degrees of all internal vertices.
\end{remark}

\section{Concluding remarks}\label{sec:conclusion}

For every positive integer $r$, we constructed infinitely many pairwise nonisomorphic caterpillars with exactly $r$ internal
vertices satisfying
\[
\chila(T)=|V(T)|=\ell(T)+r.
\]
Consequently, there is no absolute constant $C$ such that $\chila(T)\leq\ell(T)+C$ for every tree $T$, even within the class of caterpillars. This disproves the leaf conjecture and shows that its failure
cannot be remedied by replacing $2$ by any absolute constant.

The family constructed in Section~\ref{sec:unbounded} also connects this extremal behavior to classical antimagic labeling. The induced-weight intervals determined by the vertex degrees force the internal vertex weights to be pairwise distinct and greater than every leaf weight. Since the weight of each leaf is precisely the label of its unique incident edge, the leaf weights are pairwise distinct under every bijective edge labeling. Hence every bijective edge labeling of these caterpillars is antimagic. Thus the degree structure alone forces the local antimagic chromatic number to attain its largest possible value.

The positive results known for restricted families of trees motivate the search for structural conditions under which $\chila(T)\leq\ell(T)+2$ still holds or, more generally, under which $\chila(T)-\ell(T)$ remains uniformly bounded. Further questions include the characterization of trees satisfying $\chila(T)=|V(T)|$ and of those for which every bijective edge labeling is antimagic. It is also natural to determine, for each positive integer $r$, the minimum order of a tree satisfying $\chila(T)-\ell(T)=r$, both among all trees and within the class of caterpillars.

\bibliographystyle{plain} 
\bibliography{refs2}

\end{document}